\documentclass[11 pt]{article}

\usepackage[latin1]{inputenc}
\usepackage{amsmath,amsfonts,hyperref,amsthm, graphicx}
\usepackage[textsize=footnotesize,color=green!40]{todonotes}
\usepackage{fullpage}
\usepackage{authblk}
\usepackage{changes}
\usepackage{bbm}
\usepackage{wrapfig}
\usepackage{caption}
\usepackage{enumerate}

\newtheorem{theorem}{Theorem}[section]
\newtheorem{lemma}[theorem]{Lemma}
\newtheorem{fact}[theorem]{Fact}

\newtheorem{corollary}[theorem]{Corollary}
\newtheorem{conjecture}[theorem]{Conjecture}
	
\theoremstyle{definition}
\newtheorem*{definition*}{Definition}

\newtheorem{question}{Question}

\newcommand{\cC}{\mathcal{C}}

\newcommand{\cA}{\mathcal{A}}

\newcommand{\cE}{\mathcal{E}}
\newcommand{\cF}{\mathcal{F}}

\newcommand{\cS}{\mathcal{S}}

\newcommand{\cQ}{\mathcal{Q}}

\newcommand{\bE}{\mathbb{E}}

\newcommand{\bR}{\mathbb{R}}

\newcommand{\bD}{\mathbf{D}}
\newcommand{\bd}{\mathbf{d}}
\newcommand{\bX}{\mathbf{X}}

\newcommand{\sm}{\setminus}

    \makeatletter
    \let\@fnsymbol\@arabic
    \makeatother

\newcommand{\ind}{\mathbbm{1}}
 
\newcommand{\eps}{\varepsilon}

\newcommand{\bP}{\mathbb{P}}
\newcommand{\cD}{\mathcal{D}}

\newcommand{\bN}{\mathbb{N}}

\newtheoremstyle{case}{}{}{}{}{}{:}{ }{}
\theoremstyle{case}

\newcommand{\dnp}{$D(n,p)$}
\newcommand{\Var}{\mathrm{Var}}
\newcommand{\Bin}{\text{Bin}}

\author{Matthew Coulson\thanks{School of Mathematics, University of Birmingham, UK. Email: mjc685@bham.ac.uk.}}

\title{The critical window in random digraphs}

\begin{document}

\maketitle

\begin{abstract}
We consider the component structure of the random digraph $D(n,p)$ inside the critical window $p = n^{-1} + \lambda n^{-4/3}$.
We show that the largest component $\cC_1$ has size of order $n^{1/3}$ in this range.
In particular we give explicit bounds on the tail probabilities of $|\cC_1|n^{-1/3}$.
\end{abstract}

\section{Introduction}
Consider the random digraph model $D(n,p)$ where each of the  $n(n-1)$ possible edges is included with probability $p$ independently of all others.
This is analogous to the Erd\H{o}s-Renyi random graph $G(n,p)$ in which each edge is again present with probability $p$ independently of all others.
McDiarmid~\cite{mcdiarmid1980clutter} showed that due to the similarity of the two models, it is often possible to couple $G(n,p)$ and $D(n,p)$ to compare the probabilities of certain properties.

In the random graph $G(n,p)$ the component structure is well understood. In their seminal paper~\cite{erdos1960evolution}, Erd\H{o}s and R\'enyi proved that for $p = c/n$ the largest component of $G(n,p)$ has size $O(\log(n))$ if $c<1$, is of order $\Theta(n^{2/3})$ if $c=1$, and has linear size when $c>1$. This threshold behaviour is known as the double jump.
If we zoom in further around the critical point, $p=1/n$ and consider $p = (1+\eps(n))/n$ such that $\eps(n) \to 0$ and $|\eps(n)|^3 n \to \infty$, Bollob\'as~\cite{bollobas1984evolution} proved the following theorem for $|\eps|>(2\log(n))^{1/2} n^{-1/3}$,which was extended to the whole range described above by \L{}uczak~\cite{luczak1990component}.
\begin{theorem}[\cite{bollobas1984evolution, luczak1990component}]
\label{thm:pluseps}
 Let $np=1+\eps$, such that $\eps = \eps(n) \to 0$ but $n|\eps|^3 \to \infty$, and $k_0 = 2 \eps^{-2} \log(n|\eps|^3)$.
 \begin{enumerate}[i)]
  \item If $n \eps^3 \to -\infty$ then a.a.s. $G(n,p)$ contains no component of size greater than $k_0$.
  \item If $n \eps^3 \to \infty$ then a.a.s. $G(n,p)$ contains a unique component of size greater than $k_0$. This component has size $2\eps n(1+o(1))$.
 \end{enumerate}
\end{theorem}
Within the critical window itself i.e. $p = n^{-1} + \lambda n^{-4/3}$ with $\lambda \in \bR$, the size of the largest component $\cC_1$ is not tightly concentrated as it is for larger $p$.
Instead, there exists a random variable $X_1 = X_1(\lambda)$ such that $|\cC_1|n^{-2/3} \to X_1$ as $n \to \infty$.
Much is known about the distribution of $X_1$, in fact the vector $\bX = (X_1,\ldots, X_k)$ of normalised sizes of the largest $k$ components i.e. $X_i = |\cC_i| n^{-2/3}$ converges to the vector of longest excursion lengths of an inhomogeneous reflected Brownian motion by a result of Aldous~\cite{aldous1997brownian}.
In a more quantitative setting where one is more interested about behaviour for somewhat small $n$, Nachmias and Peres~\cite{nachmias2010critical} proved the following (similar results may be found in~\cite{pittel2001largest,scott2006solving}).
\begin{theorem}[\cite{nachmias2010critical}]
 Suppose $0<\delta<1/10$, $A>8$ and $n$ is sufficiently large with respect to $A, \delta$. Then if $\cC_1$ is the largest component of $G(n,1/n)$, we have
 \begin{enumerate}[i)]
  \item $\bP(|\cC_1| < \lfloor \delta n^{2/3} \rfloor) \leq 15 \delta^{3/5}$
  \item $\bP(|\cC_1|>An^{2/3}) \leq \frac{4}{A} e^{-\frac{A^2(A-4)}{32}}$
 \end{enumerate}
\end{theorem}
Note we have only stated the version of their theorem with $p=n^{-1}$ for clarity but it holds for the whole critical window.
Of course, there are a vast number of other interesting properties of $\cC_1$, see~\cite{addario2012continuum, janson1993birth, luczak1994structure} for a number of examples.

In the setting of $D(n,p)$, one finds that analogues of many of the above theorems still hold.
When working with digraphs, we are interested in the strongly connected components which we will often call the components.
Note that the weak component structure of $D(n,p)$ is precisely the component structure of $G(n,2p-p^2)$.
For $p = c/n$, Karp~\cite{karptc} and \L{}uckzak~\cite{luczak1990phase} independently showed that for $c<1$ all components are of size $O(1)$ and when $c>1$ there is a unique complex component of linear order and every other component is of size $O(1)$ (a component is complex if it has more edges than vertices).
The range $p = (1+\eps)/n$ was studied by \L{}uczak and Seierstad~\cite{luczak2009critical} who were able to show the following result which can be viewed as a version of Theorem~\ref{thm:pluseps} for $D(n,p)$,
\begin{theorem}[\cite{luczak2009critical}]
 Let $np=1+\eps$, such that $\eps = \eps(n) \to 0$.
 \begin{enumerate}[i)]
  \item If $n \eps^3 \to -\infty$ then a.a.s. every component of $D(n,p)$ is an isolated vertex or a cycle of length $O(1/|\eps|)$.
  \item If $n \eps^3 \to \infty$ then a.a.s. $D(n,p)$ contains a unique complex component of size $4 \eps^2 n (1+o(1))$ and every other component is an isolated vertex or a cycle of length $O(1/\eps)$.
 \end{enumerate}
\end{theorem}
As a corollary \L{}uczak and Seierstad obtain a number of weaker results inside the critical window regarding complex components.
They showed that there are $O_p(1)$ complex components containing $O_p(n^{1/3})$ vertices combined and that each has spread $\Omega_p(n^{1/3})$ (the \emph{spread} of a complex digraph is the length of its shortest induced path). 

Our main result is to give bounds on the tail probabilities of $|\cC_1|$ resembling those of Nachmias and Peres for $G(n,p)$.
\begin{theorem}[Lower Bound]
\label{mainlb}
 Let $0<\delta<1/800$, $\lambda \in \bR$ and $n \in \bN$. Let $\cC_1$ be the largest component of $D(n,p)$ for $p = n^{-1} + \lambda n^{-4/3}$. Then if $n$ is sufficiently large with respect to $\delta, \lambda$,
 \begin{equation}
\bP(|\cC_1| < \delta n^{1/3}) \leq 2 e \delta^{1/4}, \label{eq:mainlb}
\end{equation}
provided that $\delta \leq \frac{(\log 2)^2}{4 |\lambda|^2
}$.
\end{theorem}
Note that the constants in the above theorem have been chosen for simplicity and it is possible to give an expression for~(\ref{eq:mainlb}) depending on both $\lambda$ and $\delta$ which imposes no restriction on their relation to one another.
\begin{theorem}[Upper Bound]
\label{mainub}
 There exist constants, $\zeta, \eta >0$ such that for any $A>0, \lambda \in \bR$, if $\cC_1$ is the largest component of $D(n,p)$ for $p = n^{-1} + \lambda n^{-4/3}$. Then provided $n$ is sufficiently large with respect to $A, \lambda$,
 $$
 \bP(|\cC_1|>An^{1/3}) \leq \zeta e^{-\eta A^{3/2} + \lambda^+ A}
 $$ 
 Where $\lambda^+ = \max(\lambda, 0)$.
\end{theorem}
A simple corollary of these bounds is that the largest component has size $\Theta(n^{1/3})$. 
This follows by taking $\delta = o(1)$ in Theorem~\ref{mainlb} and $A = \omega(1)$ in Theorem~\ref{mainub}.
\begin{corollary}
 Let $\cC_1$ be the largest component of $D(n,p)$ for $p = n^{-1} + \lambda n^{-4/3}$. Then, $|\cC_1| = \Theta_p(n^{1/3})$.
\end{corollary}
It should be noted that, in contrast to the undirected case, checking whether a set of $W$ of vertices constitutes a strongly connected component of a digraph $D$ requires much more than checking only those edges with at least one end in $W$.
In particular, in order for $W$ to be a strongly connected component, it must be strongly connected and there must be no directed path starting and ending in $W$ which contains vertices that are not in $W$.
This precludes us from using a number of methods which have often been used to study $G(n,p)$.
We therefore develop novel methods for counting the number of strongly connected components of $D(n,p)$ based upon branching process arguments.

The remainder of this paper is organised as follows.
In Section~\ref{sec:enumeration} we give a pair of bounds on the number of strongly connected digraphs which have a given excess and number of vertices.
Sections~\ref{sec:lbound} and~\ref{sec:ubound} contain the proofs of Theorems~\ref{mainlb} and~\ref{mainub} respectively in the case that $p = n^{-1}$.
The proof of Theorem~\ref{mainlb} in Section~\ref{sec:lbound} is a relatively straightforward application of Janson's inequality.
The proof of Theorem~\ref{mainub} in Section~\ref{sec:ubound} is much more involved.
We use an exploration process to approximate the probability that a given subdigraph of $D(n,p)$ is also a component.
Using this we approximate the expected number of strongly connected components of size at least $An^{1/3}$ and apply Markov's inequality.
The adaptations required to handle the critical window $p = n^{-1}+\lambda n^{-4/3}$ are presented in Section~\ref{sec:adaptations}.
We conclude the paper in Section~\ref{sec:open} with some open questions and final remarks.

\section{Enumeration of Digraphs by size and excess}
\label{sec:enumeration}
For both the upper and lower bounds on the size of the largest component, we need good bounds on the number strongly connected digraphs with a given excess and number of vertices.
Where the \emph{excess} of a strongly connected digraph with $v$ vertices and $e$ edges is $e-v$.
Let $Y(m,k)$ be the number of strongly connected digraphs with $m$ vertices and excess $k$.
The study of $Y(m,k)$ was imitated by Wright~\cite{wright1977formulae} who obtained recurrences for the exact value of $Y(m,k)$. 
However, these recurrences swiftly become intractable as $k$ grows.
This has since been extended to asymptotic formulae when $k = \omega(1)$ and $O(m \log(m))$~\cite{perez2013asymptotic,pittel2013counting}.
Note that when $k = m\log(m) + \omega(n)$, the fact $Y(m,k) \sim \binom{m(m-1)}{m+k}$ is a simple corollary of a result of Pal\'asti~\cite{palasti1966strong}.
In this section we give an universal bound on $Y(m,k)$ (Lemma~\ref{uboundynk}) as well as a stronger bound for small excess (Lemma~\ref{sboundynk}). 
\begin{lemma}
\label{uboundynk}
For every $m, k \geq 1$,
 $$
 Y(m,k) \leq \frac{(m+k)^k m^{2k} (m-1)! }{k!}
 $$
\end{lemma}
\begin{proof}
 We will prove this by considering ear decompositions of the strongly connected digraphs in question.
 An \emph{ear} is a non-trivial directed path in which the endpoints may coincide (i.e. it may be a cycle with a marked start/end vertex).
 The internal vertices of an ear are those that are not endpoints.
 An \emph{ear decomposition} of a digraph $D$ is a sequence, $E_0, E_1, \ldots, E_k$ of ears such that:
 \begin{itemize}
  \item $E_0$ is a cycle
  \item The endpoints of $E_i$ belong to $\bigcup_{j=0}^{i-1} E_j$
  \item The internal vertices of $E_i$ are disjoint from  $\bigcup_{j=0}^{i-1} E_j$
  \item $\bigcup_{i=0}^{k} E_i = D$
 \end{itemize}
 We make use of the following fact.
 \begin{fact}
  A digraph $D$ has an ear decomposition with $k+1$ ears if and only if $D$ is strongly connected with excess $k$.
 \end{fact} 
 Thus we count strongly connected digraphs by a double counting of the number of possible ear decompositions.
 We produce an ear decomposition with $m$ vertices and $k+1$ ears as follows.
 First, pick an ordering $\pi$ of the vertices.
 Then insert $k$ bars between the vertices such that the earliest the first bar may appear is after the second vertex in the order; multiple bars may be inserted between a pair of consecutive vertices.
Finally, for each $i \in [k]$, we choose an ordered pair of vertices $(u_i,v_i)$ which appear in the ordering before the $i$th bar.

This corresponds to a unique ear decomposition.
The vertices in $\pi$ before the first bar are $E_0$ with its endpoint being the first vertex.
The internal vertices of $E_i$ are the vertices of $\pi$ between the $i$th and $i+1$st bar. Furthermore, $E_i$ has endpoints $u_i$ and $v_i$ and is directed from $u_i$ to $v_i$.
The orientation of every other edge follows the order $\pi$.

Hence, there are at most
$$
\binom{m+k-2}{k} m^{2k} m! \leq \frac{(m+k)^k m^{2k} m!}{k!}
$$
ear decompositions.
Note that each vertex of a strongly connected digraph is contained in a cycle.
Therefore each  vertex could be the endpoint of $E_0$ and hence at least $m$ ear decompositions correspond to each strongly connected digraph.
Hence the number of strongly connected digraphs of excess $k$ may be bounded by
$$
 Y(m,k) \leq \frac{(m+k)^k m^{2k} m!}{k!m} =\frac{(m+k)^k m^{2k} (m-1)!}{k!},
$$
as claimed. 
\end{proof}

\begin{lemma}
\label{sboundynk}
There exists $C >0$ such that for $1 \leq k \leq \sqrt{m}/3$ and $m$ sufficiently large we have,

\begin{equation}
\label{eq:simplebound}
 Y(m,k) \leq C \frac{m! m^{3k-1}}{(2k-1)!}.
\end{equation}
\end{lemma}
The proof of the above lemma follows similar lines to the proof of Theorem~$1.1$ in~\cite{perez2013asymptotic} to obtain a bound of a similar order.
We then prove that this bound implies the above which is much easier to work with.

First we introduce some definitions and notation from~\cite{perez2013asymptotic}.
A random variable $X$ has the zero-truncated Poisson distribution with parameter $\lambda>0$ denoted $X \sim TP(\lambda)$ if it has probability mass function
\begin{equation*}
\bP(X=i) = 
\begin{cases}
\frac{\lambda^i}{i! (e^\lambda -1)} & \text{if } i \geq 1, \\
0 & \text{if } i<1.
\end{cases}
\end{equation*}
Let $\cD$ be the collection of all degree sequences $\bd = (d_1^+, \ldots, d_m^+,d_1^-,\ldots,d_m^-)$ such that $d_i^+,d_i^- \geq 1$ for each $1 \leq i \leq m$ and furthermore,
$$
\sum_{i=1}^m d_i^+ = \sum_{i=1}^m d_i^-=m.
$$
A \emph{preheart} is a digraph with minimum semi-degree at least $1$ and no cycle components.
The $\emph{heart}$ of a preheart $D$ is the multidigraph $H(D)$ formed by suppressing all vertices of $D$ which have in and out degree precisely $1$.

We define the \emph{preheart configuration model}, a two stage variant of the configuration model for digraphs which always produces a preheart, as follows.
For $\bd \in \cD$, define $$T = T(\bd) = \{ i \in [m] : d_i^+ +d_i^- \geq 3 \}.$$
First we apply the configuration model to $T$ to produce a heart $H$. That is, assign each vertex $i \in T$ $d_i^+$ out-stubs and $d_i^-$ in-stubs and pick a uniformly random perfect matching between in- and out-stubs.
Next, given a heart configuration $H$, we construct a preheart configuration $Q$ by assigning $[m] \sm T$ to $E(H)$ such that the vertices assigned to each arc of $H$ are given a linear order.
Denote this assignment including the orderings by $q$.
Then the preheart configuration model, $\cQ(\bd)$ is the probability space of random preheart configurations formed by choosing $H$ and $q$ uniformly at random.
Note that each $Q \in \cQ(\bd)$ corresponds to a (multi)digraph with $m$ vertices $m+k$ edges and degree sequence $\bd$.

As in the configuration model, each simple digraph with degree sequence $\bd$ is produced in precisely $\prod_{i=1}^m d_i^+! d_i^-!$ ways.
So if we restrict to simple preheart configurations, the digraphs we generate in this way are uniformly distributed.
Where in this case, simple means that there are no multiple edges or loops (however cycles of length $2$ are allowed).
We now count the number of preheart configurations.
Let $m' = m'(\bd) = |T(\bd)|$ be the number of vertices of the heart.
Then, we have the following
\begin{lemma}
\label{preheartconfigs}
 Let $\bd \in \cD$, then there are 
 $$
 \frac{m'(\bd)+k}{m+k} (m+k)!
 $$
 preheart configurations.
\end{lemma}
\begin{proof}
 We first generate the heart, and as we are simply working with the configuration model for this part of the model, there are $(m'+k)!$ heart configurations.
 The assignment of vertices in $[m]\sm T$ to the arcs of the heart $H$ may be done one vertex at a time by subdividing any already present edge and maintaining orientation.
 In this way when we add the $i$th vertex in this stage, there are $m'+k+i-1$ choices for the edge we subdivide.
 We must add $m-m'$ edges in this stage and so there are
 $$
 \prod_{i=1}^{m-m'} m'+k+i-1 = \frac{(m+k-1)!}{(m'+k-1)!}
 $$
 unique ways to create a preheart configuration from any given heart.
 Multiplying the number of heart configurations by the number of ways to create a preheart configuration from a given heart yields the desired result.
\end{proof}
The next stage is to pick the degree sequence, $\bd \in \cD$ at random.
We do this by choosing the degrees to be independent and identically distributed zero-truncated Poisson random variables with mean $\lambda>0$.
That is, $d_i^+ \sim TP(\lambda)$ and $d_i^- \sim TP(\lambda)$ such that the family $\{ d_i^+, d_i^- : i \in [m] \}$ is independent.
Note that this may not give a degree sequence at all, or it may be the degree sequence of a digraph with the wrong number of edges.
Thus we define the event $\Sigma(\lambda)$ to be the event that
$$ 
\sum_{i=1}^m d_i^+ = \sum_{i=1}^m d_i^- = m +k.
$$
We shall now prove the following bound,
\begin{lemma}
\label{ymkbound2}
 For any $\lambda>0$ we have 
\begin{equation}
\label{eq:complexbound}
 Y(m,k) \leq \frac{3k (m+k-1)! (e^{\lambda}-1)^{2m}}{\lambda^{2(m+k)}}  \bP(\Sigma(\lambda)).
\end{equation}
\end{lemma}
\begin{proof}
 Let $\bD$ be the random degree sequence generated as above and $\bd \in \cD$, then
 \begin{equation}
 \label{eq:pDd}
 \bP(\bD =\bd) = \prod_{i=1}^m \frac{\lambda^{d_i^+}}{d_i^+!(e^\lambda -1)}\frac{\lambda^{d_i^-}}{d_i^-!(e^\lambda -1)} = \frac{\lambda^{2(m+k)}}{(e^\lambda -1)^{2m}} \prod_{i=1}^m \frac{1}{d_i^+! d_i^-!}.
 \end{equation}
 By definition of $\Sigma(\lambda)$, we have
 $$
 \sum_{\bd \in \cD} \bP(\bD=\bd) = \bP(\Sigma(\lambda)),
 $$
as all of the above events are disjoint.
Thus, we may rearrange~(\ref{eq:pDd}) to deduce that
\begin{equation}
\label{eq:ssig}
 \sum_{\bd \in \cD} \prod_{i=1}^m \frac{1}{d_i^+! d_i^-!} = \frac{(e^\lambda-1)^{2m}}{\lambda^{2(m+k)}}\bP(\Sigma(\lambda)).
\end{equation}
Lemma~\ref{preheartconfigs} tells us that for a given degree sequence $\bd$, there are
 $$
 \frac{m'(\bd)+k}{m+k} (m+k)!
 $$
 preheart configurations.
As each simple digraph with degree sequence $\bd$ comes from precisely $\prod_{i=1}^m d_i^+! d_i^-!$ configurations, and $m'(\bd) \leq 2k$ as otherwise the excess would be larger than $k$, we can deduce that the total number of prehearts with $m$ vertices and excess $k$ is
\begin {equation}
\label{eq:sumbound}
\sum_{\bd \in \cD} \frac{m'(\bd)+k}{m+k} (m+k)! \prod_{i=1}^m \frac{1}{d_i^+! d_i^-!} \leq \sum_{\bd \in \cD} (m+k)! \frac{3k}{m+k} \prod_{i=1}^m \frac{1}{d_i^+! d_i^-!}.
\end {equation}
Note that any strongly connected digraph is a preheart and so~(\ref{eq:sumbound}) is also an upper bound for $Y(m,k)$.
Finally, combining~(\ref{eq:ssig}) and~(\ref{eq:sumbound}) yields the desired inequality.
\end{proof}
It remains to prove that~(\ref{eq:complexbound}) can be bounded from above by~(\ref{eq:simplebound}).
To this end, we prove the following upper bound on $\bP(\Sigma(\lambda))$.
\begin{lemma}
\label{sigmaubound}
For $\lambda <1$,
$$
 \bP(\Sigma(\lambda)) \leq \frac{147}{\lambda m}.
$$
\end{lemma}

For the proof of this lemma, we will use the Berry-Esseen inequality for normal approximation (see for example~\cite[Section XX.2]{venkatesh2013theory}.)
\begin{lemma}
\label{berryesseen}
Suppose $X_1,X_2,\ldots, X_n$ is a sequence of independent random variables from a common distribution with zero mean, unit variance and third absolute moment $\bE|X|^3 = \gamma < \infty$.
Let $S_n = X_1+X_2+\ldots+X_n$ and let $G_n$ be the cumulative distribution function of $S_n/\sqrt{n}$.
Then for each $n$ we have
\begin{equation}
\label{tyurinbound}
 \sup_{t \in \bR} |G_n(t) - \Phi(t)| \leq \frac{\gamma}{2 \sqrt{n}},
\end{equation}
where $\Phi$ is the cumulative distribution function of the standard gaussian.
\end{lemma}

Here, the explicit constant $1/2$ in equation~(\ref{tyurinbound}) was obtained by Tyurin~\cite{tyurin2010refinement}.
\begin{proof}[Proof of Lemma~\ref{sigmaubound}]
The in-degrees of the random degree sequence are chosen independently from a truncated poisson distribution with parameter $\lambda$.
Thus, we want to apply Lemma~\ref{berryesseen} to the sum $S_m = Y_1+Y_2+\ldots+Y_m$ where the $Y_i$ are normalised truncated Poisson random variables.
So all we must compute are the first three central moments of the truncated poisson distribution.
Let $Y \sim TP(\lambda)$, one can easily compute that $\bE(Y) = c_\lambda = \frac{\lambda e^\lambda}{e^\lambda - 1}$ and $\Var(Y) =\sigma_\lambda^2= c_\lambda(1+\lambda - c_\lambda)$.
Note that for $\lambda <1$ we have $1< c_\lambda<2$ and so as $Y$ only takes integer values which are at least $1$, $\bE|Y-\bE(Y)|^3 = \bE(Y-c_\lambda)^3 + 2(c_\lambda - 1)^3 \bP(Y=1)$.
Computing this yields
\begin{equation*}
\bE|Y-\bE(Y)|^3  = \lambda + \frac{2 \lambda^4 - 5 \lambda^3 + 3 \lambda^2 - \lambda}{e^\lambda - 1} +
\frac{3 (2 \lambda^4 - 3 \lambda^3 + \lambda^2)}{(e^\lambda - 1)^2} +
\frac{2 (3 \lambda^4 - 2 \lambda^3)}{(e^\lambda - 1)^3} +
\frac{2 \lambda^4}{(e^\lambda - 1)^4}
\end{equation*}
One can check that this is bounded above by $2\lambda$ for $\lambda<1$.

The normalised version of $Y$ is $X = (Y - c_\lambda)/\sigma_\lambda$.
We have
$$
\bE|X|^3 = \bE\bigg|\frac{Y-c_\lambda}{\sigma_\lambda}\bigg|^3 = \frac{1}{\sigma_\lambda^3}\bE|Y-c_\lambda|^3 \leq \frac{2\lambda}{\sigma_\lambda^3} = \gamma.
$$
For $\lambda<1$ one can check $c_\lambda < 1+ 2\lambda/3$, which allows us to deduce that $\sigma_\lambda^2 > \lambda/3$ (also using $Y \geq 1$).
Hence, $\bE|X|^3 \leq 6\sqrt{3} \lambda^{-1/2}$.
Substituting into Lemma~\ref{berryesseen} with $G_m$ the distribution of $S_m/\sqrt{m}$, 
$$
\sup_{t \in \bR} |G_m(t) - \Phi(t)| \leq \frac{3\sqrt{3} }{\sqrt{\lambda m}}.
$$
The probability that the sum of the in-degrees is $m+k$ is precisely 
$$
G_m\bigg(\frac{m+k-m c_\lambda}{\sigma_\lambda \sqrt{m}}\bigg)-G_m\bigg(\frac{m+k-1-m c_\lambda}{\sigma_\lambda \sqrt{m}}\bigg).
$$
Following an application of the triangle inequality, we see that this probability is bounded above by
$$
\frac{6\sqrt{3}}{\sqrt{\lambda m}} + \frac{1}{\sqrt{2 \pi m} \sigma_\lambda} \leq \frac{7\sqrt{3}}{\sqrt{\lambda m}}.
$$
As the event that the in-degrees sum to $m+k$ and the event that the out-degrees sum to $m+k$ are independent and identically distributed events, we may deduce the bound,
$$
\bP(\Sigma(\lambda)) \leq \frac{147}{\lambda m}.
$$
\end{proof}
Finally, we may prove Lemma~\ref{sboundynk}.
\begin{proof}[Proof of Lemma~\ref{sboundynk}]
We choose $\lambda = 2k/m<1$ by assumption, then $\bP(\Sigma(\lambda)) \leq 147/2k$ by Lemma~\ref{ymkbound2}.
Combining this with Lemma~\ref{sigmaubound} yields
\begin{equation}
Y(m,k) \leq \frac{441 (m+k-1)!}{2} \lambda^{-2k} \bigg( \frac{e^\lambda -1}{\lambda} \bigg)^{2m} \leq \frac{441 m! m^{3k-1} e^{k^2/m}}{(2k)^{2k}}\bigg( \frac{e^\lambda -1}{\lambda} \bigg)^{2m}
\end{equation}
We use the inequality $e^x \leq 1+x+x^2/2+x^3/4$ which holds for all $0 \leq x \leq 1$ to bound $(e^\lambda -1)/\lambda \leq 1 + \lambda/2 + \lambda^2/4$.
Thus, 
$$
((e^\lambda -1)/\lambda)^{2m} \leq (1 + \lambda/2 + \lambda^2/4)^{2m} \leq e^{m\lambda + m\lambda^2/2} = e^{2k + 2k^2/m}.
$$
Then, we can use Stirling's inequality, $e \sqrt{2k-1} (2k-1)^{2k-1} e^{-2k+1} \geq (2k-1)!$, so that
$$
\frac{e^{2k}}{(2k)^{2k}} \leq \frac{e^{2k}}{(2k-1)^{2k-1/2}} \leq \frac{e^2}{(2k-1)!},
$$
allowing us to rewrite the bound on $Y(m,k)$ as
$$
Y(m,k) \leq \frac{441e^3}{2} \frac{ m! m^{3k-1}}{(2k-1)!},
$$
where we used $e^{k^2/m} \leq e^{1/3}$.
Thus proving the lemma with $C = 441e^3/2$.
\end{proof}

\section{Proof of Theorem~\ref{mainlb}}
\label{sec:lbound}
In this section we prove a lower bound on component sizes in \dnp. 
We give the proof for $p = 1/n$ for simplicity. 
The proof when $p = n^{-1} + \lambda n^{-4/3}$ is very similar, with more care taken in the approximation of terms involving $(np)^m$. See Section~\ref{sec:adaptations} for more details.
\begin{theorem}
\label{lbound}
Let $0<\delta<1/800$, then the probability that $D(n,1/n)$ has no component of size at least $\delta n^{1/3}$ is at most $2\delta^{1/2}$.
\end{theorem}
To prove this we will bound from above the probability that there is no cycle of length between $\delta n^{1/3}$ and $\delta^{1/2}n^{1/3}$.
Let $X$ be the random variable counting the number of cycles in $D(n,1/n)$ of length between $\delta n^{1/3}$ and $\delta^{1/2}n^{1/3}$.
Note that we may decompose $X$ as a sum of dependent Bernoulli random variables, and thus we may apply Janson's Inequality in the following form (see~\cite[Theorem 2.18 (i)]{jlr}).
\begin{theorem}
\label{jansonineq}
Let $S$ be a set and $S_p \subseteq S$ chosen by including each element of $S$ in $S_p$ independently with probability $p$.
Suppose that $\cS$ is a family of subsets of $S$ and for $A \in  \cS$, we define $I_A$ to be the event $\{ A \subseteq S_p \}$.
Let $\mu = \bE(X)$ and
$$
\Delta = \frac{1}{2} \mathop{\sum \sum}_{A \neq B, A \cap B \neq \emptyset} \bE(I_A I_B)
$$
Then,
$$
\bP(X = 0) \leq e^{-\mu + \Delta}
$$
\end{theorem}
To apply Theorem~\ref{jansonineq}, we define $S$ to be the set of edges of the complete digraph on $n$ vertices. Let $A \in \cS$ if and only if $A \subseteq S$ is the set of edges of a cycle of length between $\delta n^{1/3}$ and $\delta^{1/2}n^{1/3}$.
Define $X(m)$ to be the number cycles in $D(n, 1/n)$ of length $m$.
We start by approximating the first moment of $X$.
\begin{lemma}
\label{lem:mu}
 $\bE(X) \geq \log(1/\delta)/2$
\end{lemma}
\begin{proof}
 Let $a = \delta n^{1/3}$ and $b = \delta^{1/2}n^{1/3}$. Then, we can write $X$ as
 $$
 X = \sum_{m=a}^b X(m)
 $$

 Note that 
 \begin{equation}
 \bE(X(m)) = \binom{n}{m} \frac{m!}{m} p^m \geq \frac{1}{m}
 \label{eq:Xmbound}
 \end{equation}
 
 So, we may bound the expectation of $X$ as follows

 $$
 \bE(X) = \sum_{m=a}^b \bE(X(m)) \geq \sum_{m=a}^b \frac{1}{m} \geq \int_a^b \frac{dx}{x} = \frac{\log(1/\delta)}{2}
 $$
\end{proof}
Let $Z(m,k)$ be the random variable counting the number of strongly connected graphs with $m$ vertices and excess $k$ in $D(n,1/n)$.
Directly computing $\Delta$ is rather complicated so we will instead compute an upper bound on $\Delta$ that is a linear combination of the first moments of the random variables $Z(m,k)$ for $m \geq a$ and $k \geq 1$.
To move from the computation of $\Delta$ to the first moments of $Z(m,k)$ we use the following lemma,
\begin{lemma}
\label{zqk}
Each strongly connected digraph $D$ with excess $k$ may be formed in at most $27^k$ ways as the union of a pair of directed cycles $C_1$ and $C_2$.
\end{lemma}
\begin{proof}
Consider the heart $H(D)$ of $D$.
Recall that $H(D)$ is the (multi)-digraph formed by suppressing the degree $2$ vertices of $D$ and retaining orientations.
As $D$ has excess $k$, $H(D)$ has at most $2k$ vertices.
Furthermore, the excess of $H(D)$ is the same as the excess of $D$ as we only remove vertices of degree $2$.
Thus $H(D)$ has at most $3k$ edges.

Then, each edge of $H(D)$ must be a subdigraph of either $C_1$, $C_2$ or both.
So there are $3^{3k}=27^k$ choices for the pair $C_1, C_2$ as claimed.
\end{proof}
We are now in a position to give a bound on $\Delta$.
\begin{lemma}
\label{lem:Delta}
$\Delta \leq \log(2)$ for any $\delta \in (0, 1/800]$
\end{lemma}
\begin{proof}
Let
$$
\Gamma(k) := \{ E(C) | C \subseteq \overrightarrow{K_n}, C \cong \overrightarrow{C_k} \},
$$
where $\overrightarrow{K_n}$ is the complete digraph on $[n]$ and $\overrightarrow{C_k}$ is the directed cycle of length $k$.
For $\alpha \in \Gamma(k)$ let $I_\alpha$ be the indicator function of the event that all edges of $\alpha$ are present in a given realisation of $D(n,1/n)$.
Also, define 
$$
\Gamma = \bigcup_{k=a}^b \Gamma(k).
$$
Then, by definition, 
\begin{equation*}
\Delta = \frac{1}{2}\mathop{\sum\sum}_{\alpha \neq \beta, \alpha \cap \beta \neq \emptyset} \bE(I_\alpha I_\beta)
\end{equation*}
Let $\Gamma_{\alpha}^{m,k}(t)$ be the set of $\beta \in \Gamma(t)$ such that $\alpha \cup \beta$ is a collection of $m+k$ edges spanning $m$ vertices.
Then,
\begin{align}
\nonumber 2 \Delta &= \sum_{s=a}^b\sum_{t=a}^b \sum_{\alpha \in \Gamma(s)} \sum_{m=s}^{\infty} \sum_{k=1}^\infty \sum_{\beta \in \Gamma_{\alpha}^{m,k}(t)} p^{m+k} \\
\nonumber & \leq \sum_{m=a}^{2b} \sum_{k=1}^{\infty} \sum_{s = a}^{m} \sum_{t= a}^m \sum_{\alpha \in \Gamma(s)} \sum_{\beta \in \Gamma_{\alpha}^{m,k}(t)} p^{m+k} \\
\label{ex2bound} & \leq  \sum_{m=a}^{2b} \sum_{k=1}^{\infty} 27^k \bE(Z(m,k)),
\end{align}
where the last inequality follows from Lemma~\ref{zqk}.
Note that 
$$
\bE(Z(m,k)) = \binom{n}{m} p^{m+k} Y(m,k)
$$
by definition.
We will use the following two bounds on $Y(m,k)$ which follow immediately from Lemma~\ref{uboundynk}.
\begin{itemize}
\item If $k \leq m$, then $Y(m,k) \leq \frac{2^k m^{3k} m!}{k! m}$
\item If $k > m$, then $Y(m,k) \leq \frac{(2e)^k m^{2k} m!}{m}$
\end{itemize}
This allows us to split the sum in~(\ref{ex2bound}) based upon whether $k \leq m$ or $k>m$ to obtain
\begin{align}
 \nonumber 2\Delta & \leq  \sum_{m=a}^{2b} \sum_{k=1}^{m} 27^k \binom{n}{m} \frac{2^k m^{3k} m!}{k! m} p^{m+k} + \sum_{m=a}^{2b} \sum_{k=m+1}^{\infty}  27^k \binom{n}{m}  \frac{(2e)^k m^{2k} m!}{m} p^{m+k} \\
 \nonumber & \leq \sum_{m=a}^{2 b} \frac{1}{m}\sum_{k=1}^{\infty}  \frac{(54 p m^3)^k}{k!} + \sum_{m=a}^{2 b} \frac{1}{m} \sum_{k=m+1}^{\infty} (54e m^2 p)^k  \\
 & \leq \frac{\log(4/\delta)}{2}(e^{432 \delta^{3/2}}-1 +23328e^2\delta^2) \label{eq:Delta}
\end{align}
Where the $23328e^2\delta^2$ term comes from noting $k \geq 2$ in the range $k \geq m+1$ and that for $x\leq 1/2$
$$
\sum_{k=2}^\infty x^k \leq 2x^2
$$
As~(\ref{eq:Delta}) is increasing in $\delta$, we simply need to check that the Lemma holds for $\delta = 1/800$ which may be done numerically.

\end{proof}
Finally, to prove Theorem~\ref{lbound} we substitute the values obtained for $\mu$ and $\Delta$ in Lemmas~\ref{lem:mu} and~\ref{lem:Delta} respectively into Theorem~\ref{jansonineq}.
That is,
$$
\bP(X = 0) \leq e^{-\mu +\Delta} \leq e^{-\log(1/\delta)/2 + \log(2)} = 2 \delta^{1/2}
$$
So the probability there is no directed cycle of length at least $\delta n^{1/3}$ is at most $2 \delta^{1/2}$ and, as cycles are strongly connected, this is also an upper bound on the probability there is no strongly connected component of size at least $\delta n^{1/3}$.
\section{Proof of Theorem~\ref{mainub}}
\label{sec:ubound}
In this section we prove an upper bound on the component sizes in $D(n,p)$.
Again, we only consider the case when $p=1/n$ to simplify notation and calculations.
The reader is referred to Section~\ref{sec:adaptations} for a sketch of the adaptations to extend the result to the full critical window.
The following is a restatement of Theorem~\ref{mainub} for $p=1/n$.
\begin{theorem}
\label{thm:upperbound}
 There exist constants $\zeta, \eta > 0$ such that for any $A>0$ if $n$ is sufficiently large with respect to $A$, then the probability that $D(n,1/n)$ contains any component of size at least $An^{1/3}$ is at most
 $
 \zeta e^{-\eta A^{3/2}}.
 $
\end{theorem}
We will use the first moment method to prove this theorem and calculate the expected number of large strongly connected components in $D(n,1/n)$.
Note that it is important to count components and not strongly connected subgraphs as the expected number of strongly connected subgraphs in $D(n,1/n)$ blows up as $n \to \infty$.
Thus for each strongly connected subgraph, we will use an exploration process to determine whether or not it is a component.

The exploration process we use was initially developed by Martin-L\"of~\cite{martinlof} and Karp~\cite{karptc}.
During this process, vertices will be in one of three classes: \emph{active}, \emph{explored} or \emph{unexplored}.
At time $t \in \bN$, we let $X_t$ be the number of active vertices, $A_t$ the set of active vertices, $E_t$ the set of explored vertices and $U_t$ the set of unexplored vertices.

We will start from a set $A_0$ of vertices of size $X_0$ and fix an ordering of the vertices, starting with $A_0$. 
For step $t \geq 1$, if $X_{t-1}>0$ let $w_t$ be the first active vertex.
Otherwise, let $w_t$ be the first unexplored vertex.
Define $\eta_t$ to be the number of unexplored out-neighbours of $w_t$ in $D(n,1/n)$.
Change the class of each of these vertices to active and set $w_t$ to explored.
This means that $|E_t| = t$ and furthermore, $|U_t| = n - X_t - t$.
Let $N_t = n - X_t - t - \ind(X_t = 0)$ be the number of potential unexplored out-neighbours of $w_{t+1}$ i.e. the number of unexplored vertices which are not $w_{t+1}$.
Then, given the history of the process, $\eta_t$ is distributed as a binomial random variable with parameters $N_{t-1}$ and $1/n$.
Furthermore, the following recurrence relation holds.
\begin{equation}
 X_t =
 \begin{cases}
 X_{t-1} + \eta_t -1 & \text{if } X_{t-1} > 0, \\
 \eta_t  & \text{otherwise}
 \end{cases}
\end{equation}
Let $\tau_1 = \min\{t \geq 1 : X_t = 0\}$. Note that this is a stopping time and at time $\tau_1$ the set $E_{\tau_1}$ of explored vertices is precisely the out-component of $A_0$.
If $A_0$ spans a strongly connected subdigraph $D_0$ of $D(n,1/n)$, then $D_0$ is a strongly connected component if and only if there are no edges from $E_{\tau_1} \sm A_0$ to $A_0$.
The key idea will be to show that if $X_0$ is sufficiently large, then it is very unlikely for $\tau_1$ to be small, and consequently it is also very unlikely that there are no edges from $E_{\tau_1} \sm A_0$ to $A_0$.
This is encapsulated in the following lemma.
\begin{lemma}
\label{lem:exprocessbound}
 Let $X_t$ be the exploration process defined above with starting set of vertices $A_0$ of size $X_0=m$. Suppose $0<c<\sqrt2$ is a fixed constant. Then,
 $$
 \bP(\tau_1 < c m^{1/2}n^{1/2}) \leq 2 e^{-\frac{(2-c^2)^2}{8c} m^{3/2}n^{-1/2} +O(m^2 n^{-1})}.
 $$
\end{lemma}
\begin{proof}
 Define $\xi = c  m^{1/2}n^{1/2}$ and consider the auxiliary process, $X_t'$ which we define recursively by
 \begin{align*}
  X_0' & = m,\\
  X_t' & = X_{t-1}' - 1 + W_t \text{ for } t \geq 1,
 \end{align*}
where $W_t \sim \Bin(n-t-10m,p)$.
Let $\tau_2$ be the stopping time,
$$
\tau_2 = \inf \{ t : X_t>10 m \}
$$
We may couple the processes $(X_t, X_t')$ such that $X_t'$ is stochastically dominated by $X_t$ for $t < \tau_2$.
The coupling may be explicitly defined by setting $\eta_t = W_t + W_t'$ with $W_t' \sim \Bin(10 m - X_{t-1},p)$.
Define another stopping time, $\tau_1' = \min \{t \geq 1 : X_t' =0 \}$. Consider the following events
\begin{align*}
\cE_1 & = \{ \tau_1  < c m^{1/2}n^{1/2} \} \\
\cE_2 & = \{ \tau_1'  < c m^{1/2}n^{1/2} \} \\
\cE_3 & = \{ \tau_2 < c m^{1/2}n^{1/2} \}
\end{align*}
And note that $\bP(\cE_1) \leq \bP(\cE_2) + \bP(\cE_3)$ by our choice of coupling and a union bound (as the coupling guarantees $\cE_1 \subseteq \cE_2 \cup \cE_3$). Thus we only need to bound the probabilities of the simpler events $\cE_2$ and $\cE_3$.
We begin by considering $\cE_3$. To bound its probability we consider the upper bound process $M_t$ defined by
\begin{align*}
 M_0 & = m, \\
 M_t & = M_{t-1} - 1 +B_t \text{ for } t \geq 1,
\end{align*}
where $B_t \sim \Bin(n,1/n)$. It is straightforward to couple $(X_t,M_t)$ such that $M_t$ stochastically dominates $X_t$.
Furthermore, $M_t$ is a martingale.
Hence, $\bP(\cE_3) \leq \bP(\tau_2' < c  m^{1/2}n^{1/2})$ where $\tau_2'$ is the stopping time, $\tau_2' = \min\{t:M_t>10m\}$.
To bound the probability of $\cE_2$ consider the process $Y_t$ defined as $Y_t = m - X_t'$.
One can check that $Y_t$ is a submartingale.

As $x \mapsto e^{\alpha x}$ is a convex non-decreasing function for any $\alpha>0$, we may apply Jensen's inequality to deduce that $Z_t^- = e^{\alpha Y_t}$ and $Z_t^+ = e^{\alpha M_t}$ are submartingales.
Also, $Z_t^-,Z_t^+ >0$ for any $i \in \bN$.
Starting with $Z_t^-$, we may apply Doob's maximal inequality~\cite[Section~12.6]{grimmett2001probability} and deduce that
\begin{equation}
\label{eq:applyDMI}
 \bP\bigg(\min_{0 \leq t \leq \xi} X_i' \leq 0 \bigg) = \bP\bigg(\max_{0 \leq t \leq \xi} Z_t^-\geq e^{\alpha m}\bigg)  \leq \frac{\bE(Z_{\xi}^-)}{e^{\alpha m}} 
\end{equation}
We may rewrite this by noting that 
$$
Y_t = m - X_t' = t - \sum_{i=1}^t W_i = t - R_t
$$
where $R_t$ is binomially distributed and in particular $R_\xi \sim \Bin(l\xi,p)$ for 
$$
l\xi = c m^{1/2}n^{3/2} - \frac{c^2 mn}{2} -10 c m^{3/2}n^{1/2} + \frac{c m^{1/2} n^{1/2}}{2} 
$$
Also, we choose $x$ such that $x l \xi = \xi - m$.
Then~(\ref{eq:applyDMI}) may be rewritten as $e^{-\alpha m} \bE(Z_\xi^-) = e^{\alpha xl \xi} \bE(e^{-\alpha R_\xi})$.
The next stage is to rearrange this into a form which resembles the usual Chernoff bounds (for $x<p$).
So, let
$$
f(\alpha) = e^{\alpha xl \xi} \bE(e^{-\alpha R_\xi}) = \bigg[e^{\alpha x} (p e^{-\alpha} + 1 - p)\bigg]^{l \xi}
$$
Then, we choose $\alpha^*$ to minimise $f$. Solving $f'(\alpha)=0$, we obtain the solution
$$
e^{-\alpha^*} = \frac{x(1-p)}{p(1-x)}
$$
Note $x<p$ so, $e^{-\alpha^*}<1$ and $\alpha^*>0$ as desired.
Thus,
\begin{align*}
 f(\alpha^*) = & = \bigg[\bigg(  \frac{p(1-x)}{x(1-p)}\bigg)^{x} \bigg( x \frac{1-p}{1-x} + 1-p \bigg)\bigg]^{mt} \\
& = \bigg[\bigg( x \frac{1-p}{1-x} + 1-p \bigg) \bigg( \frac{p}{x}\bigg)^x \bigg( \frac{1-p}{1-x} \bigg)^x \bigg]^{mt} \\
& = \bigg[\bigg( \frac{p}{x}\bigg)^x \bigg( \frac{1-p}{1-x} \bigg)^{1-x} \bigg]^{mt}
\end{align*}
Which is the usual expression found in Chernoff bounds.
As usual, we bound this by writing
$$
f(\alpha^*) = e^{-g(x)l \xi}
$$
and bound $g$, where
$$
g(x) = x \log\bigg(\frac{x}{p}\bigg) + (1-x) \log\bigg(\frac{1-x}{1-p}\bigg)
$$
Computing the Taylor expansion of $g$ we find that $g(p)=g'(p)=0$.
So, if $g''(x) \geq \beta$ for all $x$ between $p$ and $p-h$, then $g(p-h) \geq \beta h^2/2$. Furthermore,
$$
g''(x) = \frac{1}{x} + \frac{1}{1-x}
$$
As $0<x<p$, we have $g''(x) \geq 1/x \geq 1/p$.
So, we deduce that $g(x) \geq \delta^2 p/2$ where $\delta = 1-x/p$. 
All that remains is to compute $\delta$.
As defined earlier, we have $x l \xi = \xi - m$ which for convenience we will write as
\begin{equation}
x l \xi = \xi\bigg(1-\frac{m^{1/2}}{c n^{1/2}} \bigg) \label{eq:xlxi}
\end{equation}
Also, as $p=n^{-1}$, and recalling the definition of $l \xi$ from earlier, 
\begin{align}
\nonumber pl \xi & =  c m^{1/2}n^{1/2} - \frac{c^2 m}{2} +O(m^{3/2} n^{-1/2})  \\
  & = \xi\bigg(1 - \frac{cm^{1/2}}{2n^{1/2}} +O(m n^{-1}) \bigg) \label{eq:plxi}
\end{align}
We divide~(\ref{eq:xlxi}) by~(\ref{eq:plxi}) and as the Taylor expansion of $1/(1-w)$ is $\sum_{i \geq 0} w^i$,
\begin{equation}
 \frac{x}{p} = \frac{1-\frac{m^{1/2}}{c n^{1/2}}}{1 - \frac{cm^{1/2}}{2n^{1/2}} +O(m n^{-1})} = 1 - \frac{m^{1/2}}{c n^{1/2}} + \frac{cm^{1/2}}{2n^{1/2}} +O(m n^{-1})
\end{equation}
From which we may deduce
\begin{equation}
 \delta = \frac{(2-c^2)m^{1/2}}{2c n^{1/2}}  + O(m n^{-1})
\end{equation}
So,
\begin{equation}
\bP(\cE_2) \leq e^{-\frac{\delta^2 p}{2} l \xi} = e^{-\frac{(2-c^2)^2}{8c} m^{3/2}n^{-1/2} +O(m^2 n^{-1})} \label{eq:pe2}
\end{equation}
We may proceed similarly for $Z_t^+$, in particular we must still appeal to Doob's maximal inequality as we seek a bound over the entire process.
In this case we end up with a $\Bin(n \xi, p)$ distribution and are looking at the upper tail rather than the lower.
We find $p n \xi = \xi$ and
$$
x n \xi = \xi + 9 m = \xi \bigg(1 + \frac{9 m^{1/2}}{c n^{1/2}}\bigg)
$$
Thus,
$$
\delta = \frac{x}{p}-1 = \frac{9 m^{1/2}}{cn^{1/2}}
$$
Substituting into the analogous bound, 
\begin{equation}
 \bP(\cE_3) \leq e^{-\frac{\delta^2 p}{3} n \xi} \leq e^{- \frac{27 m^{3/2}}{c n^{1/2}}} \label{pe3}
\end{equation}
Observe that $\bP(\cE_2) \geq \bP(\cE_3)e^{O(m^2 n^{-1})}$ for $0<c<\sqrt{2(1+3\sqrt{6})}$.
Thus, in the range we are interested in, we may use $2\bP(\cE_2)$ as an upper bound for $\bP(\cE_2)+\bP(\cE_3)$ and this proves the lemma.
\end{proof}
We now compute the probability that any given strongly connected subgraph of $D(n,1/n)$ is a component.
To do so, we use the simple observation that a strongly connected subgraph is a component if it is not contained in a larger strongly connected subgraph.
\begin{lemma}
\label{lem:sccpbty}
 There exist $\beta, \gamma>0$ such that if $H$ is any strongly connected subgraph of $D(n,1/n)$ with $m$ vertices.
 Then the probability that $H$ is a strongly connected component of $D(n,1/n)$ is at most $\beta e^{-(1+\gamma)m^{3/2}n^{-1/2} +O(m^2 n^{-1})}$.
\end{lemma}
\begin{proof}
 We compute the probability that $H$ is a component of $D(n,1/n)$ by running the exploration process $X_t$ starting from $A_0 = V(H)$.
 So, $X_0 = m$.
 Once the exploration process dies at time $\tau_1$, any backward edge from $E_{\tau_1} \sm A_0$ to $A_0$ gives a strongly connected subgraph of $D(n,1/n)$ which contains $H$.
 Let $Y_t$ be the random variable which counts the number of edges from $E_{\tau_1} \sm A_0$ to $A_0$.
 Note that for $t \geq m$, $Y_t \sim \Bin(m(t-m),p)$.
 Furthermore, $H$ is a strongly connected component of $D(n,1/n)$ if and only if $Y_{\tau_1}=0$.
 
 Let $\eps>0$ and define the events $\cA_i$ for $i =1, \ldots, r$ (where $r \sim c/\eps$ for some $c>1$) to be
 \begin{align*}
  \cA_i & = \{ (i-1) \eps m^{1/2}n^{1/2} \leq \tau_1 < i \eps m^{1/2}n^{1/2} \},\\
  \cA_{r+1} & = \{ r \eps m^{1/2} n^{1/2} \leq \tau_1 \}.
 \end{align*}
Clearly the family $\{ \cA_i : i = 1, \ldots, r +1 \}$ forms a partition of the sample space.
So, by the law of total probability,
\begin{equation}
\label{eq:lotp}
 \bP(Y_{\tau_1}=0) = \sum_{i=1}^{r+1} \bP(Y_{\tau_1}=0 |\cA_i) \bP(\cA_i)
\end{equation}
By applying Lemma~\ref{lem:exprocessbound} when $1\leq i \leq r$ we find
$$
\bP(\cA_i) \leq 2 e^{-\frac{(2-i^2\eps^2)^2}{8i\eps}m^{3/2}n^{-1/2}+O(m^2n^{-1})}
$$
Note that $Y_{\tau_1}$ conditioned on $\cA_i$ stochastically dominates a $\Bin(m((i-1) \eps m^{1/2}n^{1/2} - m),p)$ distribution.
Therefore,
$$
\bP(Y_{\tau_1}=0|\cA_i) \leq (1-p)^{m((i-1) \eps m^{1/2}n^{1/2} - m)} \leq e^{-(i-1)\eps m^{3/2}n^{-1/2} + O(m^2 n^{-1})}
$$
Combining the above and substituting into~(\ref{eq:lotp}) yields
\begin{align}
\bP(Y_{\tau_1}=0) & \leq 2 \sum_{i=1}^{r} e^{-((i-1)\eps+\frac{(2-i^2\eps^2)^2}{8i\eps})m^{3/2}n^{-1/2}+O(m^2n^{-1})} + e^{-r\eps m^{m/2}n^{-1/2}+O(m^2n^{-1})}, \label{secondterm} \\
& \leq (2r+1) e^{-(1+\gamma)m^{3/2}n^{-1/2}+O(m^2n^{-1})},
\end{align}
for some $\gamma>0$ provided that $\eps$ is sufficiently small.
The second term in~(\ref{secondterm}) is a result of the fact $\bP(A_{r+1}) \leq 1$.
This proves the lemma and if one wishes for explicit constants, taking $\eps = 0.025$, $r= 45$ works and gives $\beta < 100 $, $\gamma > 0.06 $.
\end{proof}

The next stage in our proof is to show that a typical instance of $D(n,1/n)$ has no component of large excess and no exceptionally large components.
This will allow us to use the bound from Lemma~\ref{sboundynk} to compute the expected number of large strongly connected components of $D(n,1/n)$.
The first result in this direction is an immediate corollary of a result of \L{}uczak and Seierstad~\cite{luczak2009critical}.
\begin{lemma}[\cite{luczak2009critical}]
\label{lem:nogc}
 The probability that $D(n,1/n)$ contains a strongly connected component of size at least $n^{1/3} \log\log n$ is $o_n(1)$.
\end{lemma}
The next lemma ensures that there are not too many cycles which enables us to prove that the total excess is relatively small.
\begin{lemma}
\label{lem:nolongcyc}
 The probability that $D(n,p)$ contains more than $n^{1/6}$ cycles of length bounded above by $n^{1/3}\log\log(n)$ is $o_n(1)$.
\end{lemma}
\begin{proof}
In this proof and subsequently we will use the convention that $\log^{(k)} x$ is the logarithm function composed with itself $k$ times, while $(\log x)^k$ is its $k$th power.
 We shall show that the expected number of cycles of length at most $n^{1/3} \log^{(2)}n$ is $o(n^{1/6})$ at which point we may apply Markov's inequality.
 So let $C$ be the random variable which counts the number of cycles of length at most $n^{1/3} \log^{(2)} n$ in $D(n,1/n)$. We can calculate its expectation as
\begin{equation}
\label{eq:cyclesbound1}
\bE(C) = \sum_{k=1}^{n^{1/3} \log^{(2)}n} \binom{n}{k} \frac{k!}{k} p^k \leq \sum_{k=1}^{n^{1/3} \log^{(2)}n} \frac{1}{k}
\end{equation}
We use the upper bound on the $k$th harmonic number $H_k \leq \log k +1$, which allows us to deduce that 
\begin{equation}
\label{eq:cyclesbound2}
\bE(C) \leq H_{n^{1/3} \log^{(2)} n} \leq \frac{1}{3}\log n + \log^{(3)} n  +1 \leq \log n = o(n^{1/6}).
\end{equation}
Thus the lemma follows by Markov's inequality.
\end{proof}
\begin{corollary}
\label{cor:nobigxs}
 The probability that $D(n,1/n)$ contains a component of excess at least $n^{1/6}$ and size at most $n^{1/3} \log\log n $ is $o_n(1)$.
\end{corollary}
\begin{proof}
 If $D$ is any strongly connected digraph with $m$ vertices and excess $k$, then note that it must have at least $k+1$ cycles of length at most $m$.
 This can be seen by considering the ear decomposition of $D$.
 The first ear must be a cycle, and each subsequent ear adds a path which must be contained in a cycle as $D$ is strongly connected.
 So as we build the ear decomposition, each additional ear adds at least one cycle.
 As any ear decomposition of a strongly connected digraph of excess $k$ has $k+1$ ears, then $D$ must have at least $k+1$ cycles.
 
 Thus, if $D$ has $k$ cycles, it must have excess at most $k-1$.
 So applying Lemma~\ref{lem:nolongcyc} completes the proof.
\end{proof}
Finally, we prove the main theorem of this section.
\begin{proof}[Proof of Theorem~\ref{thm:upperbound}]
 Let $\cC_1$ be the largest strongly connected component of $D(n,1/n)$ and $L_1 =|\cC_1|$.
 We want to compute $\bP(L_1 \geq A n^{1/3})$.
 Define the following three events,
\begin{align*}
\cE_1 & = \{ L_1 \geq A n^{1/3} \} \\
\cE_2 & = \{ A n^{1/3} \leq L_1 \leq n^{1/3} \log\log(n) \} \\
\cE_3 & = \{ L_1 \geq n^{1/3} \log\log(n) \}
\end{align*}
Clearly, $\cE_1 \subseteq \cE_2 \cup \cE_3$ and by Lemma~\ref{lem:nogc}, $\bP(\cE_3) = o_n(1)$.
If $\cF$ is the event that $\cC_1$ has excess at least $n^{1/6}$ then by Corollary~\ref{cor:nobigxs}, $\bP(\cE_2 \cap \cF) = o_n(1)$.
All that remains is to give a bound on $\bP(\cE_2 \cap \cF^c)$.
To this end let $N(A)$ be random variable which counts the number of strongly connected components of $D(n,1/n)$ which have size between $An^{1/3}$ and $n^{1/3}\log\log n$ and excess bounded above by $n^{1/6}$.
By Markov's inequality, we may deduce that $\bP(\cE_2 \cap \cF^c) \leq \bE(N(A))$.
Computing the expectation of $N(A)$,
\begin{equation}
 \label{eq:eNA1}
 \bE(N(A)) = \sum_{m=An^{1/3}}^{n^{1/3}\log^2(n)} \sum_{k=0}^{n^{1/6}} \binom{n}{m} p^{m+k} Y(m,k) \bP(Y_{\tau_1}=0|X_0=m).
\end{equation}
In Lemma~\ref{lem:sccpbty} we showed that $\bP(Y_{\tau_1}=0|X_0=m) \leq \beta e^{-(1+\gamma)m^{3/2} n^{-1/2} +O(m^2 n^{-1})}$. 
Also, using Lemma~\ref{sboundynk} we can check that
\begin{equation}
 \sum_{k=0}^{n^{1/6}} Y(m,k) p^k \leq (m-1)! + C(m-1)!(m^3 p)^{1/2} \sinh((m^3p)^{1/2}), \label{eq:ymkp}
\end{equation}
where the first term on the right hand side of~(\ref{eq:ymkp}) comes from the directed cycles and $C$ is the same constant as in Lemma~\ref{sboundynk}.
As $\sinh(x) \leq e^x$ we can bound~(\ref{eq:ymkp}) by
\begin{align*}
 \sum_{k=0}^{n^{1/6}} Y(m,k) p^k & \leq (m-1)! (1+C m^{3/2}n^{-1/2} e^{m^{3/2}n^{-1/2}}) \\
 & \leq 2(m-1)! C m^{3/2}n^{-1/2} e^{m^{3/2}n^{-1/2}}
\end{align*}
Combining these bounds and using $\binom{n}{m} \leq n^m/m!$ we deduce
\begin{align}
\nonumber \bE(N(A)) & \leq \sum_{m=An^{1/3}}^{n^{1/3}\log^2(n)} \bigg(\frac{(np)^m}{m!}\bigg)\bigg(2(m-1)! C m^{3/2}n^{-1/2} e^{m^{3/2}n^{-1/2}}\bigg)\bigg(\beta e^{-(1+\gamma)m^{3/2} n^{-1/2} +O(m^2 n^{-1})}\bigg) \\
\nonumber & = \sum_{m=An^{1/3}}^{n^{1/3}\log^2(n)} \frac{2 \beta C m^{1/2}}{n^{1/2}} e^{-\gamma m^{3/2}n^{-1/2} + O(m^2n^{-1})} \\
& \leq \int_{m=An^{1/3}}^{n^{1/3}\log^2(n)+1} \frac{2 \beta C m^{1/2}}{n^{1/2}} e^{-\frac{\gamma}{2} m^{3/2}n^{-1/2}} dm \label{eq:integral1}
\end{align}
where~(\ref{eq:integral1}) holds for all sufficiently large $n$.
Now making the substitution $x = m n^{-1/3}$ we can remove the dependence of~(\ref{eq:integral1}) on both $m$ and $n$ so that
\begin{align}
\nonumber \bE(N(A))& \leq 2 \beta C \int_A^{\log^2(n) + n^{-1/3}} x^{1/2} e^{-\frac{\gamma}{2} x^{3/2}} dx \\
& \leq 2 \beta C \int_A^{\infty} x^{1/2} e^{-\frac{\gamma}{2} x^{3/2}} dx \nonumber\\
& = \frac{8 \beta C}{3 \gamma} \int_{\frac{\gamma A^{3/2}}{2}}^\infty e^{-t}dt = \frac{8 \beta C}{3 \gamma} e^{-\frac{\gamma A^{3/2}}{2}} \label{eq:magic}
\end{align}
So, by Markov's inequality $\bP(\cE_2 \cap \cF^c) \leq \zeta e^{-\eta A^{3/2}}$ where $\zeta$ and $\eta$ are the corresponding constants found in~(\ref{eq:magic}). So, 
$$
\bP(L_1 \geq An^{1/3}) \leq \bP(\cE_2 \cap \cF^c) +\bP(\cE_2 \cap \cF)+\bP(\cE_3) = \zeta e^{-\eta A^{3/2}} + o_n(1).
$$
Calculating $\zeta$ and $\gamma$ using the values for $C, \beta$ and $\gamma$ in Lemmas~\ref{sboundynk} and~\ref{lem:sccpbty} yields $\zeta < 2 \times 10^7$ and $\eta > 0.03$.
\end{proof}

\section{Adaptations for the Critical Window}\label{sec:adaptations}
In this section we sketch the adaptations one must make to the proofs of Theorems~\ref{lbound} and~\ref{thm:upperbound} such that they hold in the whole critical window, $p = n^{-1} + \lambda n^{-4/3}$ where $\lambda \in \bR$.
\subsection{Lower Bound}
For Theorem~\ref{lbound}, the adaptation is rather simple.
We will still apply Janson's inequality and so we only need to recompute $\mu$ and $\Delta$.
Furthermore, the only difference in these calculations comes from replacing the term $n^{-m-k}$ by $p^{m+k}$, and in fact the $p^k$ in this turns out to make negligible changes.
In this light, Lemma~\ref{lem:mu} changes to
\begin{lemma}
 \begin{equation*}
 \label{lem:admu}
\bE(X) \geq 
 \begin{cases}
  -e^{\frac{\lambda \delta}{2}}\log(\delta)/2 & \text{if }\lambda \geq 0 \\
  -e^{2 \delta^{1/2} \lambda} \log(\delta)/2 & \text{otherwise}
 \end{cases}
\end{equation*}
\end{lemma}
where the only difference in the proof is to bound $(1+\lambda n^{-1/3})^m$ by its lowest value depending on whether $\lambda \geq 0$ or $\lambda <0$.
We bound this via
\begin{equation*}
 1+x \geq
 \begin{cases}
  e^{\frac{x}{2}} & \text{if } 0 \leq x \leq 2 \\
  e^{2x} & \text{if } -\frac{1}{2} \leq x \leq 0
 \end{cases}
\end{equation*}
Furthermore, Lemma~\ref{lem:Delta} changes to
\begin{lemma}
\label{lem:adDelta}
For all sufficiently large $n$ and small enough $\delta$,
 \begin{equation*}
\Delta \leq 
 \begin{cases}
 e^{2\delta^{1/2}\lambda} \log(2) & \text{if } \lambda \geq 0 \\
 e^{\delta \lambda} \log(2) & \text{otherwise}
 \end{cases}
\end{equation*}
\end{lemma}
The proof again is almost identical with the only change being to approximate the $(np)^m$ term.
This time we seek an upper bound so use the approximation $1+x \leq e^x$ which is valid for any $x$.
We still need to split depending upon the sign of $\lambda$ as for the above constants we upper bound $(np)^m$ by its largest possible value over the range $\delta n \leq m \leq 2 \delta^{1/2} n$.
Combining Lemmas~\ref{lem:admu} and~\ref{lem:adDelta} with the relevant constraints on $\delta$ in relation to $\lambda$ yields Theorem~\ref{mainlb}.
\subsection{Upper Bound}
There is no significant (i.e. of order $e^{\lambda A}$) improvement which can be made with our current method of proof when $\lambda <0$.
This is because the gains we make computing the expectation in the proof of Theorem~\ref{thm:upperbound} are cancelled out by losses in the branching process considerations of Lemma~\ref{lem:exprocessbound}.

When $\lambda>0$ we cannot simply use our bound for $p=n^{-1}$ and thus an adaptation is necessary.
Note that by monotonicity in $p$, the results of Lemmas~\ref{lem:exprocessbound} and~\ref{lem:sccpbty} remain true for $p = n^{-1} + \lambda n^{-4/3}$ with $\lambda>0$.
The next adaptation which must be made is in equation~(\ref{eq:cyclesbound1}) where now, the expectation becomes
\begin{equation*}
 \bE(\cC) \leq \sum_{k=1}^{n^{1/3} \log^{(2)} n} \frac{e^{k \lambda n^{-1/3}}}{k} \leq \sum_{k=1}^{n^{1/3} \log^{(2)} n} \frac{(\log n)^\lambda}{k} \leq 2 (\log n)^{\lambda+1} = o(n^{1/6})
\end{equation*}
Thus allowing us to deduce the result of Corollary~\ref{cor:nobigxs} as before.
Finally all that remains is to conclude the proof of Theorem~\ref{mainub}.
Ignoring lower order terms, the only difference to the proof compared to that of Theorem~\ref{thm:upperbound} is in the computation of $\bE(N(A))$ where we must change the term $(np)^m$.
Thus the integral in~(\ref{eq:integral1}) becomes
\begin{equation}
 \int_{m=An^{1/3}}^{n^{1/3}\log^{(2)} n+1} \frac{2 \beta C m^{1/2}}{n^{1/2}} e^{-\frac{\gamma}{2} m^{3/2}n^{-1/2}+ \lambda mn^{-1/3}} dm \label{eq:ad:integral1}
\end{equation}
This is much more complex than before due to the extra term in the exponent.
However we are still able to give a bound after making the obvious substitution $t = \frac{\gamma}{2} m^{3/2}n^{-1/2}- \lambda mn^{-1/3}$, we obtain
\begin{align}
\nonumber \bE(N(A)) & \leq \frac{8 \beta C}{3 \gamma} \int_{\frac{\gamma}{2}A^{3/2} - \lambda A}^{\infty} \frac{m^{1/2}n^{-1/2}}{m^{1/2}n^{-1/2} - \frac{4 \lambda n^{-1/3}}{3 \gamma}} e^{-t} dt \\
& \leq \frac{10 \beta C}{3 \gamma} \int_{\frac{\gamma}{2}A^{3/2} - \lambda A}^{\infty} e^{-t} dt  = \frac{10 \beta C}{3 \gamma} e^{-\frac{\gamma}{2} A^{3/2} +\lambda A} \label{eq:ad:integral2}
\end{align}
which is of the claimed form.
Note the second inequality holds for $A$ sufficiently large compared to $\lambda$.

\section{Concluding Remarks}
\label{sec:open}
In this paper we have proven that inside the critical window, $p=n^{-1}+\lambda n^{-4/3}$, the largest component of $D(n,p)$ has size $\Theta_p(n^{1/3})$. Furthermore, we have given bounds on the tail probabilities of the distribution of the size of the largest component.
Combining this result with previous work of Karp~\cite{karptc} and \L{}uczak~\cite{luczak1990phase} allows us to deduce that $D(n,p)$ exhibits a ``double-jump'' phenomenon at the point $p=n^{-1}$.
However, there are still a large number of open questions regarding the giant component in $D(n,p)$.
Perhaps the most obvious such question is to ask for an exact distribution for the size of the giant component.
\begin{question}
 What is the limiting distribution of $n^{-1/3}|\cC_1(D(n,p))|$ when $p = n^{-1} + \lambda n^{-4/3}$?
\end{question}
Given the strong connection between $G(n,p)$ and $D(n,p)$, it seems likely that the limit distributions, $X^\lambda = n^{-2/3}|\cC_1(G(n,p))|$ and $Y^\lambda = n^{-1/3}|\cC_1(D(n,p))|$ (where $p=n^{-1} + \lambda n^{-4/3}$) are closely related.
For larger $p$, previous work~\cite{karptc, luczak1994structure} has found that the size of the giant strongly connected component in $D(n,p)$ is related to the size of the square of the giant component in $G(n,p)$. That is, if $|\cC_1(G(n,p)| \sim \alpha(n) n$, then $|\cC_1(D(n,p)| \sim \alpha(n)^2 n$. 
Note that the result found in Theorem~\ref{mainub} is consistent with this pattern as here we have an exponent of order $A^{3/2}$ while for $G(n,p)$ a similar result is true with exponent $A^3$ implying that the probability we find a component of size $Bn^{2/3}$ in $G(n,p)$ is similar to the probability of finding a component of size $B^2 n^{1/3}$ in $D(n,p)$ (assuming both bounds are close to tight).
As such, we make the following conjecture to explain this pattern.
\begin{conjecture}
 If $X^\lambda$ and $Y^\lambda$ are the distributions defined above and $X_1^\lambda, X_2^\lambda$ are independent copies of $X^\lambda$ then, $Y^\lambda = X_1^\lambda X_2^\lambda$.
 \label{conj:sclimit}
\end{conjecture}
Note that recently, Goldschmidt and Stephenson~\cite{goldschmidt2019scaling} found a scaling limit for the sizes of all the strong components in the critical random digraph.
This is analogous to the result of Aldous~\cite{aldous1997brownian} in $G(n,p)$.
This result, while giving a limit object which one can work with seems to be difficult to extract any information on the exact distribution of $n^{-1/3}|\cC_1(D(n,p))|$, and so still leaves conjecture~\ref{conj:sclimit} open.

Finally, we consider the transitive closure of random digraphs.
The \emph{transitive closure} of a digraph $D$ is $cl(D)$ a digraph on the same vertex set as $D$ and such that $uv$ is an edge of $cl(D)$ if and only if there is a directed path from $u$ to $v$ in $D$.
Equivalently, $cl(D)$ is the smallest digraph containing $D$ such that the relation $R$ defined by $uRv$ if and only if $uv$ is an edge is transitive.
Karp~\cite{karptc} gave a linear time algorithm to compute the transitive closure of a digraph from the model $D(n,p)$ provided that $p \leq (1-\eps)n^{-1}$ or $p \geq (1+\eps)n^{-1}$.
For all other $p$ this algorithm runs in time $O(f(n) (n \log n)^{4/3})$ where $f(n)$ is any $\omega(1)$ function.
Now that we know more about the structure of $D(n,p)$ for $p$ close to $n^{-1}$, it may be possible to adapt Karp's algorithm and obtain a better time complexity.
\begin{question}
 Does there exist a linear time algorithm to compute the transitive closure of $D(n,p)$ when $(1-\eps)n^{-1} \leq p \leq (1 + \eps)n^{-1}$?
 \end{question}
 
\section{Acknowledgements}
The author would like to thank Guillem Perarnau for providing helpful feedback on a previous draft of this paper. 

\bibliographystyle{plain}
{\small \bibliography{DNP_bib}}

\end{document}